\documentclass[a4paper,10pt,reqno]{amsart}
\usepackage{amsfonts,amssymb,amscd,amsmath,enumerate,verbatim,calc}
\usepackage[all]{xy}

\newcommand{\CM}{Cohen-Macaulay}

\newcommand{\B}{\mathcal{B} }

\newcommand{\n}{\mathfrak{n} }
\newcommand{\m}{\mathfrak{m} }

\newcommand{\Ac}{\mathcal{A} }
\newcommand{\Bc}{\mathcal{B} }
\newcommand{\C}{\mathcal{C} }
\newcommand{\D}{\mathcal{D} }
\newcommand{\T}{\mathcal{T} }

\newcommand{\Z}{\mathbb{Z} }

\newcommand{\rt}{\rightarrow}

\newcommand{\ov}{\overline}

\newcommand{\wh}{\widehat }
\newcommand{\wt}{\widetilde }

\newcommand{\CMS}{\operatorname{\underline{CM}}}

\newcommand{\Ass}{\operatorname{Ass}}

\newcommand{\mmod}{\operatorname{mod}}

\newcommand{\coh}{\operatorname{coh}}
\newcommand{\im}{\operatorname{im}}

\newcommand{\thick}{\operatorname{thick}}

\newcommand{\proj}{\operatorname{proj}}
\newcommand{\Supp}{\operatorname{Supp}}

\newcommand{\rad}{\operatorname{rad}}
\newcommand{\Hom}{\operatorname{Hom}}
\newcommand{\sHom}{\operatorname{\underline{Hom}}}

\theoremstyle{plain}

\newtheorem{theorem}{Theorem}[section]
\newtheorem{corollary}[theorem]{Corollary}
\newtheorem{lemma}[theorem]{Lemma}
\newtheorem{proposition}[theorem]{Proposition}

\theoremstyle{definition}
\newtheorem{definition}[theorem]{Definition}
\newtheorem{s}[theorem]{}
\newtheorem{remark}[theorem]{Remark}

\theoremstyle{remark}

\begin{document}

\title[KRS-categorical]{Categorical Krull-Remak-Schmidt for triangulated categories}

\author{Tony~J.~Puthenpurakal}

\address{Department of Mathematics, IIT Bombay, Powai, Mumbai 400 076}

\email{tputhen@math.iitb.ac.in}
\date{\today}

 \begin{abstract}
 Let $R$ be a commutative ring
If $\C_1$ and $\C_2$ are $R$-linear triangulated categories then we can give an obvious triangulated structure on $\C = \C_1 \oplus \C_2$ where $\Hom_\C(U, V) = 0$ if $U \in \C_i$ and $V \in \C_j$ with $i \neq j$. We say a $R$-linear triangulated category $\C$ is disconnected if $\C = \C_1 \oplus \C_2$ where $\C_i$ are non-zero triangulated subcategories of $\C$.
Let $\C_i$ and $\D_j$ be connected  triangulated $R$ categories with $i \in \Gamma$ and $j \in \Lambda$.
Suppose there is an equivalence of triangulated $R$-categories
\[
\Phi \colon \bigoplus_{i \in \Gamma}\C_i  \xrightarrow{\cong} \bigoplus_{j \in \Lambda}\D_j
\]
Then we show that  there is a bijective function $\pi \colon \Gamma \rt \Lambda$ such that we have an equivalence $\C_i \cong \D_{\pi(i)} $ for all $i  \in \Gamma$.
We give several examples of connected triangulated categories and also of triangulated subcategories which decompose into utmost finitely many components.
\end{abstract}
 \maketitle
\section{Introduction}
We use \cite{N} for notation on triangulated categories. However we will assume that if $\mathcal{C}$ is a triangulated category then $\Hom_\mathcal{C}(X, Y)$ is a set for any objects $X, Y$ of $\mathcal{C}$. Throughout $R$ is some commutative ring. All our additive categories will be $R$-categories ( for this notion see \cite[p.\ 28]{ARS}).

\subsection{} \emph{Direct Sum of triangulated categories:} \\
Let $\C, \D$ be additive $R$-categories. By $\T = \C \oplus \D$ we mean a category with
\begin{enumerate}
  \item  Objects of $\T$ are direct sums $U \oplus V$ where $U$ is an object of $\C$ and $V$ is an object in $\D$.
  \item Morphism are defined by $\Hom_\T(U_1\oplus V_1, U_2 \oplus V_2) = \Hom_\C(U_1, U_2)\oplus \Hom_\D(V_1, V_2)$.
\end{enumerate}
It is clear that $\T$ is an additive $R$-category.

If $\C, \D$  are triangulated $R$-categories then we can define a triangulated structure on $\T$ as follows:
First define the shift operator on $\T$ as $\Sigma_T  = \Sigma_\C \oplus \Sigma_\D$. We say a sequence
\[
U_1 \oplus V_1 \xrightarrow{(f_1, g_1)}  U_2 \oplus V_2 \xrightarrow{(f_2, g_2)} U_3 \oplus V_3 \xrightarrow{(f_3, g_3)} \Sigma_\C{U_1} \oplus \Sigma_\D(V_1),
\]
to be a distinguished triangle in $\T$ if
\[
U_1  \xrightarrow{f_1}  U_2  \xrightarrow{f_2} U_3 \xrightarrow{f_3} \Sigma_\C{U_1} ,
\]
is a distinguished triangle in $\C$ and
\[
V_1  \xrightarrow{g_1}  V_2  \xrightarrow{g_2} V_3 \xrightarrow{g_3} \Sigma_\D{V_1} ,
\]
is a distinguished triangle in $\D$.
It is a routine exercise that  these distinguished triangles give a triangulated structure on $\T$. Note $\C, \D$ are thick subcategories of $\T$.

Analogously one can define a triangulated structure on $ \bigoplus_{i \in \Gamma} \C_i$ where $\C_i$ are triangulated categories for all $i \in \Gamma$.

\begin{definition}\label{conn-cat}
  Let $\C$ be a triangulated category. We say $\C$ is \emph{connected} if
  \begin{enumerate}
    \item $\C \neq 0$.
    \item If $\C = \C_1 \oplus \C_2$ where $\C_1, \C_2$ are triangulated categories then either $\C_1 = 0$ or $\C_2 = 0$.
  \end{enumerate}
\end{definition}

Our categorical Krull-Remak-Schmidt theorem is as follows:
\begin{theorem}\label{KRS}
Let $\C_i$ and $\D_j$ be connected  triangulated $R$ categories with $i \in \Gamma$ and $j \in \Lambda$.
Suppose there is an equivalence of triangulated $R$-categories
\[
\Phi \colon \bigoplus_{i \in \Gamma}\C_i  \xrightarrow{\cong} \bigoplus_{j \in \Lambda}\D_j
\]
Then  there is a bijective function $\pi \colon \Gamma \rt \Lambda$ such that we have an equivalence $\C_i \cong \D_{\pi(i)} $ for all $i  \in \Gamma$.
\end{theorem}

Theorem \ref{KRS} is interesting only if we can give a good stock of examples of triangulated categories which are a finite sum of connected triangulated categories.
We show that many triangulated categories which arise in theory of Artin Algebras, commutative algebra and algebraic geometry can be decomposed into a finite direct sum of connected triangulated categories.

A natural question is whether a triangulated category decomposes into a direct sum of connected triangulated categories. In this regard we prove the following result
\begin{theorem}\label{decomp}
Let $\C$ be a Krull-Remak-Schmidt triangulated category. Assume the class of thick subcategories of $\C$ is a set. Then $\C$ decomposes as a direct sum of connected triangulated categories.
\end{theorem}
\begin{remark}\label{skel}
  If $\C$ is a skeletally small  triangulated category then note that the class of thick subcategories of $\C$ is a set.
\end{remark}
Perhaps the assumption on $\C$ to be KRS is bit strong. We also prove
\begin{theorem}\label{decomp-A}
Let $(A,\m)$ be a Noetherian local ring.
Let $\C$ be an $A$-linear triangulated category such that $\Hom_\C(X, Y)$ is a finitely generated $A$-module for any $X, Y \in \C$. Assume the class of thick subcategories of $\C$ is a set. Then $\C$ decomposes as a direct sum of connected triangulated categories.
\end{theorem}
As a consequence of Theorem \ref{decomp-A} and remark \ref{skel} we obtain
\begin{corollary}
  Let $k$ be a field.
Let $\C$ be a Hom-finite $k$-linear triangulated category. Assume  $\C$ is skeletally small. Then $\C$ decomposes as a direct sum of connected triangulated categories.
\end{corollary}
Here is an overview of the contents of this paper. In section two we give a proof of Theorem \ref{KRS}. In the next section we give some examples of triangulated categories which decompose into finitely many subcategories. In section four we give some examples of connected categories. In section five we prove Theorems \ref{decomp} and \ref{decomp-A}. Finally in the last section we give an example of a triangulated category which decomposes into infinitely many components.
\section{Proof of Theorem \ref{KRS}}
In this section we give
\begin{proof}[Proof of Theorem \ref{KRS}]
Let $U \subseteq \bigoplus_{j \in \Lambda}\D_j $ be a \emph{connected} thick subcategory.\\
\emph{Claim-1:} $U \subseteq \D_j$ for some $j$.\\
Let $U_i$ be the full subcategory consisting of objects in $\D_i$ which are contained in $U$.
Then $U_i$ is triangulated. We assert that
$$U  = \bigoplus_{j \in \Lambda}U_j.$$
Indeed it is clear that $U  \supseteq \bigoplus_{j \in \Lambda}U_j.$
 To see the reverse inclusion let $X$ be an object in $U$.
Then $X = X_1 \oplus X_2 \oplus \cdots \oplus X_m$ where $X_i $ is an object in $\D_{j_i}$. As $U$ is thick it contains all direct summands of its objects. So $X_i \in U_{j_i}$. Thus our assertion holds.
As $U$ is connected it follows that $U = U_j$ for some $j$ and $U_i =  0$ for $i \neq j$. This proves  Claim=1.

Fix $i$. Let $\im \Phi_i $ be the essential image of $\C_i$ in  $\bigoplus_{j \in \Lambda}\D_j$. Then by our claim $ \im_\Phi  \subseteq \D_{j}$ for some $j$. As $\Phi$ is an equivalence we get $\Phi$ restricts to an equivalence $\C_i \rt \im \Phi_i$.\\
\emph{Claim-2:} $\im \Phi_i =  \D_j$ .\\
Indeed consider the inverse $\Phi^{-1}$. Notice $\Phi^{-1}(\D_j) \cap \C_i \neq 0$.  It follows from Claim-1 that $\Phi^{-1}(\D_j) \subseteq \C_i$. Claim-2  follows.

Set $\pi(i) = j$. It follows that $\pi$ is bijective and $\C_i \cong \D_{\pi(i)}$ as triangulated categories.
\end{proof}
Our proof of Theorem \ref{KRS} also shows the following
\begin{lemma}\label{conn-sum}
Let $U$ be a connected triangulated category. Suppose $U \subseteq \D$ is thick in $\D$. Suppose $\D  =  \bigoplus_{j \in \Lambda}\D_j$. Then $U \subseteq \D_i$ for some $i$.
\end{lemma}
\begin{remark}
  In Lemma \ref{conn-sum} we do not insist that $\D_i$ are also connected.
\end{remark}
\section{Examples}
We give a large class of examples of triangulated categories $\C$ where we can show that $\C$ is a direct sum of utmost finitely many connected categories. It takes some effort to prove whether a  triangulated category is connected.

\textbf{I}. If a triangulated category $\C$ has no proper thick subcategories then clearly $\C$ is connected. \\
Examples:
\begin{enumerate}
  \item Let $(A,\m)$ be a Noetherian local ring and let $K^b_f(\proj A)$ be the bounded  homotopy category of finitely generated projective $A$-modules with finite length cohomology. Then $K^b_f(\proj A)$ has no proper thick subcategories, see \cite[1.2]{N-kproj}

  \item Let $(A,\m)$ be a hypersurface singularity (i.e., the completion $\wh{A} = Q/(f)$ where $(Q,\n)$ is regular local and $f \in \n^2$). Let $\CMS(A)$ be the stable category of maximal \CM \ $A$-modules and let $\CMS^0(A)$ be the thick subcategory of MCM $A$-modules which are free on the punctured spectrum of $A$. Then $\CMS^0(A)$ has no proper thick subcategories, \cite[6.6]{T-thick}
\end{enumerate}

\textbf{II}  Let $\C$ be a Krull-Remak-Schmidt triangulated category. Let $X \in \C$ be non-zero. Then clearly $\D = \thick(X)$ is utmost a direct sum of $r$ subcategories where $r$ is the number of indecomposable summands of $X$.
Examples :
\begin{enumerate}
  \item Let $A$ be a $R$-Artin algebra where $R$ is a commutative Artin ring.  Let $D^b(\mmod(A))$ be the bounded derived category of $\mmod(A)$; the category of all finitely generated $A$-modules. Then it is well-known that $D^b(\mmod(A))$ is a Krull-Remak-Schmidt triangulated category. Furthemore $D^b(\mmod(A)) = \thick(A/\rad(A))$ where $\rad(A)$ is the radical of $A$, see \cite[7.37]{RR}. In particular if $A$ is local then $D^b(\mmod(A))$ is connected.
  \item Let $(A,\m)$ be a Henselian Gorenstein local ring of dimension $d$.  Let $\CMS(A)$ be the stable category of maximal \CM \ $A$-modules and let $\CMS^0(A)$ be the thick subcategory of MCM $A$-modules which are free on the punctured spectrum of $A$. Then $\CMS^0(A) = \thick(\Omega^d_A(k))$ where $\Omega^d_A(k)$ is the $d^{th}$-syzygy of $k$, see \cite[2.6, 2.9]{T-thick}. Thus $\CMS^0(A)$ is a finite direct sum of connected subcategories. If $\Omega^d(k)$ is indecomposable then $\CMS^0(A)$ is connected. Note when $d = 0$ clearly $\Omega^0_A(k) = k$ is indecomposable. If the multiplicity of $A$ is at least three and $d = 1, 2$ then $\Omega^d_A(k)$ is indecomposable, see \cite[Theorems A, B]{Taka}.
\end{enumerate}

\textbf{III:} Let      $\Ac$ be an Abelian category. Let $\Bc$ be a Krull-Remak-Schmidt  Serre-subcategory of $\Ac$. Let $D^b_\Bc(\Ac)$ be the bounded derived category of $\Ac$ with cohomology in $\B$. Let $X \in D^b_\Bc(\Ac)$ and let $r = $ number of irreducible summands of $H^*(X)$. Then it is clear that $X$ cannot be a direct sum of $r + 1$ non-zero elements in $\D^b_\Bc(\Ac)$. It follows that $\thick(X)$ is a finite direct sum of up to $r$ connected summands.
     Examples:
     \begin{enumerate}
       \item Let $V$ be a projective variety over an algebraically closed field $k$. Let $\Ac = \coh(V)$ be the  category of coherent sheaves on $V$. Then $\Ac$ is a Krull-Schmidt category, see \cite{At}.
         By \cite[7.38]{RR},  $D^b(\Ac) = \thick(E)$ for some $E$. Thus $D^b(\Ac)$ is a finite direct sum of connected triangulated categories.
       \item Let $A$ be a Noetherian ring. Let $\Ac = \mmod(A)$ and let $\Bc$ be its Serre-subcategory consisting of modules of finite length.
     \end{enumerate}

\textbf{IV:} Let $A$ be a Noetherian ring of finite global dimension. Then $D^b(\mmod(A)) = \thick(A)$.  If $A$ is indecomposable then $D^b(\mmod(A))$ is connected.
\begin{enumerate}
  \item If $A$ is a commutative domain then $A$ is indecomposable. So $D^b(\mmod(A))$ is connected.
  \item If $A = A_n(K)$ the $n^{th}$-Weyl Algebra over a field $K$ of characteristic zero. By \cite{Roos} $A$ has finite global dimension. We claim that $A$ is indecomposable. Suppose if possible $A$ is decomposable. Say $A = U \oplus V$ where $U, V$ are left modules. Let $\mathcal{F}$ be the Bernstein filtration on $A$. Then $G_\mathcal{F}(A) = K[X_1,\ldots, X_n, \ov{\partial_1}, \cdots, \ov{\partial_n}]$ is a polynomial ring in $2n$-variables, see \cite[I.2.2]{Bjork}. By considering the induced filtration on $U$ we get that its associated is a submodule of $G_\mathcal{F}(A)$ which is a domain. So $\dim U = 2n$. Similarly $\dim V = 2n$. So the multiplicity $e(A) = e(U) + e(V) \geq 2$, a contradiction since multiplicity of $A$ is one. Thus $A$ is indecomposable. So $D^b(\mmod A)$ is connected.
\end{enumerate}

\section{Some examples of connected categories}
In this section we prove that some natural triangulated categories are connected.

\subsection{}\label{bridge} Let $X$ be a Noetherian scheme. Let $D^b(X)$ be the bounded derived category ofcoherent sheaves on $X$. Then
$D^b(X)$ is connected  if and only if $X$ is connected, see \cite[3.2]{Br}.

\subsection{} The technique to show some triangulated categories   are connected is based on the following:
\begin{lemma}\label{conn-lemm} Let $\C$ be a triangulated category. Let $\D$ be a triangulated subcategory of $\C$ such that
\begin{enumerate}[\rm (1)]
  \item $\D$ is thick.
  \item $\D$ is connected.
  \item If $\mathcal{T} $ is a non-zero thick subcategory of $\C$ then $\D \cap \mathcal{T} \neq 0$.
\end{enumerate}
Then $\C$ is also connected.
\end{lemma}
\begin{proof}
  Suppose if possible $\C$ is not connected. So $\C = \C_1\oplus \C_2$ where $\C_i \neq 0$. By \ref{conn-sum} we get that $\D \subseteq \C_i$ for some $i$. But by (3) of our hypotheses $\D \cap \C_j \neq 0$ for both $j = 1, 2$, a contradiction.
\end{proof}

We first prove
\begin{proposition}\label{der-local}
Let $(A,\m)$ be a Noetherian local ring.  Let $D^b(\mmod(A))$ be the bounded derived category of $\mmod(A)$; the abelian  category of all finitely generated $A$-modules. Then $D^b(\mmod(A))$ is connected.
\end{proposition}
\begin{remark}
This also follows from \ref{bridge} as $Spec(A)$ is connected. However we give a purely algebraic proof for the convenience of algebraists reading this paper.
\end{remark}
We now give:
\begin{proof}[Proof of \ref{der-local}]
We identify $D^b(\mmod(A)) $ with $\C = K^{b,*}(\proj A)$ the category of \\  bounded above complexes of finitely generated projective $A$-modules with bounded cohomology.  Let $\C_f$ be complexes in $\C$ with finite length cohomology. It is clear that $\C_f$ is a thick subcategory of $\C$.

We first show that $\C_f$ is connected. Suppose if possible  $\C_f$ is not connected. Say $\C_f = \D_1 \oplus \D_2$ where $\D_i \neq 0$. Let $X$ be a minimal free resolution of $k = A/\m$. Then $X$ is indecomposable. Say $X \in \D_1$. Let $Y \in \D_2$ be non-zero. We may assume $Y$ is a minimal complex. After shifting we may assume $Y^i = 0$ for $i > 0$ and $Y^0 \neq 0$. As $Y$ is a minimal complex we get $Z = H^0(Y) \neq 0$. Note $Z$ is a non-zero finite length module. Let $\epsilon \colon Z \rt k$ be any surjective map. Then $\epsilon$ has a lift $\wt{\epsilon} \colon Y \rt X$. Note $\wt{\epsilon} \neq 0$ as the map on zeroth cohomology is non-zero. This contradicts the fact that $\Hom(D_1, D_2) = 0 $ for $D_i \in \D_i$. Thus $\C_f$ is connected.

For $Z \in \C$ let $\Supp(Z) = \bigcup_{i \in \Z} \Supp_A H^i(Z)$. Set $\dim Z = \dim \Supp(Z)$.

Let $\T$ be a non-zero thick subcategory of $\C$. We show $\C_f \cap \T \neq 0$.
 Let $Y \in \T$ be such that
$$ \dim Y = \min\{ \dim Z \mid Z \in \T \ \text{and} \ Z \neq 0 \}. $$
We assert $\dim Y = 0$. If not choose $x \in \m$ such that
$$ \ker \{  H^*(Y) \xrightarrow{x} H^*(Y) \}  \ \quad \text{has finite length.}$$
Let $Y \xrightarrow{x} Y \rt Z \rt Y[1]$ be a triangle. Then $Z \in \T$ as $\T$ is thick. Note by construction $\dim Z \leq \dim Y -1$. Note if $Z = 0$ then the map $H^*(Y) \xrightarrow{x} H^*(Y)$
is surjective and so by Nakayama Lemma $H^*(Y) = 0$ and so $Y= 0$, a contradiction.  Thus we have $Z \in \T$, $Z \neq 0$ with $\dim Z < \dim Y$, which is a contradiction. So $\dim Y = 0$. Thus
$Y \in \C_f \cap \T$ and $Y \neq 0$.

By Lemma \ref{conn-lemm} we get that $\C$ is connected.
\end{proof}
The following result can be shown in a similar manner to \ref{der-local}. So we only sketch a proof.
\begin{proposition}
  \label{hom-connected}
  Let $(A,\m)$ be a Noetherian local ring. Let $K^b(\proj A)$ be the homotopy category of bounded complexes of finitely generated free $A$-modules. Then $K^b(\proj A)$ is connected.
\end{proposition}
\begin{proof}
  \emph{Sketch} Let $K^b_f(\proj A)$ be the bounded  homotopy category of finitely generated projective $A$-modules with finite length cohomology. Then $K^b_f(\proj A)$ has no proper thick subcategories, see \cite[1.2]{N-kproj}. So $K^b_f(\proj A)$ is connected. Also clearly $K^b_f(\proj A)$ is thick in $K^b(\proj A)$.

   An argument similar to one in \ref{der-local} shows that if $\T$ is a non-zero thick subcategory of $K^b(\proj A)$ then $K^b_f(\proj A) \cap \T \neq 0.$

   Thus by Lemma \ref{conn-lemm} we get that $K^b(\proj A)$ is connected.

\end{proof}
Let $(A,\m)$ be a local Gorenstein ring. Let $\CMS(A)$ be the stable category of maximal \CM \ $A$-modules. Let $\CMS^0(A)$ be the subcategory of MCM $A$-modules which are free on the punctured spectrum of $A$. Then it is easy to check that $\CMS^0(A)$ is a thick subcategory of $\CMS(A)$.
We show
\begin{proposition}
  \label{stable} If $\CMS^0(A)$ is connected then $\CMS(A)$ is connected.
\end{proposition}
We first give an application of Proposition \ref{stable}.
\begin{corollary}
\label{hyp}
Let $(A,\m)$ be a local Gorenstein ring. Let $\CMS(A)$ be the stable category of maximal \CM \ $A$-modules. Then  $\CMS(A)$ is connected in the following cases:
\begin{enumerate}[\rm (1)]
  \item $A$ is a hypersurface singularity.
  \item $\dim A = 1, 2$ and multiplicity of $A$ is $\geq 3$.
\end{enumerate}
\end{corollary}
\begin{proof}
  In all the cases we have $\CMS^0(A)$ is connected, see I.(2) and II.(2) in the previous section. So the result follows from \ref{stable}.
\end{proof}
We now give
\begin{proof}[Proof of Proposition \ref{stable}]
By \ref{conn-lemm} it suffices to show that if $\T$ is a non-zero thick subcategory of $\CMS(A)$ then $\T \cap \CMS^0(A) \neq 0$.

For $M \in \CMS(A)$ let
$$\Supp(M) = \{ P \mid P \ \text{is prime and} \ M_P \neq 0 \ \in \CMS(A_P) \}. $$
It is easily shown that $\Supp(M) = \Supp(\sHom(M,M))$.

Let $Y \in \T$ be non-zero such that
$$ \dim  \Supp(Y)  = \min\{ \dim \Supp(Z) \mid Z \in \T \ \text{and} \ Z \neq 0 \}. $$
We assert that $\dim  \Supp(Y) = 0$. If not let
$$x \in \m \setminus \left(\bigcup_{P \in \Ass(A)}P \cup \bigcup_{Q \text{minimal in}\ \Supp(Y)} Q \right). $$
Note $x$ is a non-zero divisor of $A$ ( and hence of $Y$). We have a triangle
$$ Y \xrightarrow{x} Y \rt Z \rt Y[1].$$
It is evident that $\dim \Supp(Z) \leq \dim \Supp(Y) -1$. Furthermore note $Z \neq 0$ in $\CMS(A)$ (i.e., $Z$ is not free as an $A$-module, for otherwise we have an exact sequence
$0 \rt Q \rt Z \rt Y/xY \rt 0$ with $Q$-free (this follows from the triangle structure in $\CMS(A)$). So $Y/xY$ has finite projective dimension. As $x$ is $Y$-regular it follows that $Y$ has finite projective dimension. So $Y$ is free as an $A$-module and thus $Y = 0$ in $\CMS(A)$, a contradiction. As $\T$ is thick we get $Z\in \T$. This is a contradiction by our choice of $Y$.
Thus  $\dim \Supp(Y) = 0$, in other words $Y \in \CMS^0(A)$. So $\T \cap \CMS^0(A) \neq 0$.
\end{proof}
\section{Proof of Theorems \ref{decomp} and \ref{decomp-A}}
In this section we  give proofs of Theorems \ref{decomp} and \ref{decomp-A}. We first give
\begin{proof}[Proof of Theorem \ref{decomp}]
Let
$$\Lambda  = \{ \D \mid \D \ \text{is a connected thick subcategory of $\C$} \}.$$
First note that $\Lambda$ is a set. Let $X \in \C$ be indecomposable. Then we have $\thick(X)$ is connected. So $\Lambda$ is non-empty.
We define a partial order on $\Lambda$ with $\D \leq \D^\prime$ if $\D \subseteq \D^\prime$.

We first assert that $\Lambda$ has maximal elements. For that we use Zorn's Lemma. Let $\{ \D_\alpha \}_{ \alpha \in \Delta}$ be a chain in $\Lambda$.
Set $$ \D = \bigcup_{\alpha \in \Delta} \D_\alpha. $$
Then it is easy to see that $\D$ is a thick triangulated subcategory of $\C$. We show $\D$ is connected. Suppose if possible $\D$ is disconnected. Say $\D = \D_1 \oplus \D_2$.
where $\D_1 \neq 0$ and $\D_2 \neq 0$. Let $\D_\gamma \cap \D_1 \neq 0$. Then by \ref{conn-cat} it follows that $\D_\gamma \subseteq \D_1$. Similarly  as $\D_2 \neq 0$ it follows that $\D_\beta \subseteq \D_2$. As $\{ \D_\alpha \}_{ \alpha \in \Delta}$ is  a chain it follows that $\D_\gamma \subseteq \D_\beta$ or $\D_\beta \subseteq \D_\gamma$. Either case yields that $\D_1 \cap \D_2 \neq  \phi$, a contradiction. So $\D$ is connected. Thus $\Lambda$ has maximal elements.

Let $X$ be indecomposable. Set
$$\Lambda_X  = \{ \D \mid \D \ \text{is a connected thick subcategory of $\C$ and $\D \supseteq \thick(X)$} \}.$$
Then $\thick(X) \in \Lambda_X$. We give the obvious partial order of $\Lambda_X$ and similarly prove that $\Lambda_X$ has maximal elements. Also note a maximal element in $\Lambda_X$ is also maximal in $\Lambda$.

Let $\D, \D^\prime $ be maximal elements in $\Lambda$. As $\thick(\D, \D^\prime)$ is disconnected it follows that $\thick(\D \oplus \D^\prime) = \D \oplus \D^\prime$.
An easy induction
yields that if $\D_1, \ldots, \D_n$ are maximal elements in $\Lambda$ then
$$ \thick(\D_1 \oplus \cdots \oplus \D_n)  = \D_1 \oplus \cdots \oplus \D_n.$$
Consider
$$\D = \bigoplus_{\stackrel{\D_\alpha}{ \D_\alpha \text{maximal in} \ \Lambda}} \D_\alpha.$$
Then $\D$ is clearly a thick subcategory of $\C$. Let $X \in \C$ be indecomposable. Then $\thick(X) \subseteq \D^\prime$ for some maximal element in $\Lambda_X$. As noted before $\D^\prime $ is maximal in $\Lambda$. So $X \in \D$.
Therefore $\D = \C$. The result follows.
\end{proof}
Next we give
\begin{proof}[Proof of Theorem \ref{decomp-A}]
We first prove:\\
Claim: Let $X \in \C$. Then $X$ is a finite direct sum of indecomposables in $\C$. \\
Proof of Claim: If $X$ is indecomposable then we have nothing to show. Otherwise $X = U_1 \oplus U_2$. If $U_1, U_2$ are indecomposable then we are done. Iterating we may assume
$X = U_1 \oplus \cdots \oplus U_n$.  Note that $\bigoplus_{n \geq 1} \Hom_\C(U_i, U_i)$ is a direct summand of $\Hom_\C(X, X)$. If $E$ is a finitely generated $A$-module set $\mu(E) = \dim_A E/\m E$. Notice $\mu(\Hom_\C(X, X)) \geq n$. Thus if $\mu(\Hom_\C(X,X)) = r$ then $X$ decomposes as a direct sum of utmost $r$ indecomposables.

We remark that if $X$ is indecomposable then $\thick(X)$ is connected.
Rest of the proof is similar to proof of Theorem \ref{decomp}.
\end{proof}
\section{A triangulated category with infinitely many components}
We give an example of a triangulated category which decomposes into infinitely many components.

\s Let $k$ be a field and let $A = k[X_1, \ldots, X_d]$. Let $\D = D^b(\mmod A) = K^b(\proj A)$ be the bounded derived category of $A$. We note that $A$ has infinitely many maximal ideals.
Set
$$\C = \D_f = \{ X \in K^b(\proj A) \mid H^*(X) \ \text{has finite length} \}. $$
If $\m$ is a maximal ideal in $A$ then set
$$\C_\m = \{ X \in \C \mid \ \Supp(H^*(X)) \subseteq \m\}. $$
We claim that if $X \in \C_\m$ and $Y \in \C_\n$ where $\m$ and $\n$ are distinct maximal ideals then $\Hom_\C(X, Y) = 0$. To see this simply localize at a maximal ideal $P$ of $A$. Then $X_P = 0$ or $Y_P = 0$. Thus
$$\D = \thick ( \C_\m \colon  \m \ \text{a maximal ideal of $A$} ) = \bigoplus_{\m \in MaxSpec(A)}\C_\m.$$

To prove our result it suffices to show $\D = \C$. We do this by induction on $\ell(H^*(X))$. We may assume that $H^0(X) \neq 0$ and $X^i = 0$ for $i > 0$. If $\ell(H^*(X)) = 1$ then note $H^0(X) = A/\m$ for some maximal ideal $\m$ of $A$. It follows that $X \in \C_\m \subseteq \D$.
Next we assume that the result is known for all complexes with $\ell(H^*(X)) < n$ and we prove the result when $\ell(H^*(X)) = n$.
We may assume that $X^i = 0$ for $i > 0$ and $H^0(X) \neq 0$. Let $\m \in \Supp(H^0(X))$. Then $H^0(X)/\m H^0(X) \neq 0$ and so we have a surjection $H^0(X) \rt A/\m$. So we have a chain map
$f \colon X \rt Y$ where $Y$ is a projective resolution of $A/\m$. We note that $\ell(H^*(Cone(f))) = n -1$ and so by induction hypothesis $Cone(f) \in \D$. Also $Y \in \D$. As $\D$ is thick we have $X \in \D$. The result follows.

\emph{Acknowledgment:} We thank Martin Kalck for many discussions on this paper.

\end{document}